\documentclass[11pt,reqno]{amsart}

\usepackage[T1]{fontenc}
\usepackage[utf8]{inputenc}
\usepackage[margin=1in]{geometry}
\usepackage{amsmath,amssymb}
\usepackage{xcolor}
\usepackage{microtype}
\usepackage{indentfirst}
\usepackage[
  colorlinks=true,
  linkcolor={blue!75!black},
  citecolor={blue!80!black},
  urlcolor={blue!65!black}
]{hyperref}

\numberwithin{equation}{section}

\theoremstyle{plain}
\newtheorem{theorem}{Theorem}[section]
\newtheorem{lemma}[theorem]{Lemma}

\theoremstyle{definition}
\newtheorem{definition}[theorem]{Definition}

\theoremstyle{remark}
\newtheorem{remark}[theorem]{Remark}

\newcommand{\N}{\mathbb{N}}
\newcommand{\Z}{\mathbb{Z}}
\newcommand{\F}{\mathbb{F}}
\newcommand{\SU}{\operatorname{SU}}
\newcommand{\PSU}{\operatorname{PSU}}
\newcommand{\CT}{\operatorname{CT}}
\newcommand{\UT}{\operatorname{UT}}
\newcommand{\rank}{\operatorname{rank}}
\newcommand{\ord}{\operatorname{ord}}
\newcommand{\lcm}{\operatorname{lcm}}
\DeclareMathOperator{\supp}{supp}

\hypersetup{
  pdftitle={On Some Problems from the Kourovka Notebook},
  pdfauthor={Wouter van Doorn, Elias Judin, Pietro Monticone, and Daniel Morrison},
  pdfsubject={Solutions to eight problems from the Kourovka Notebook, formally verified in Lean},
  pdfkeywords={Kourovka Notebook, group theory, formal verification, Lean, Mathlib,
    automated theorem proving},
  pdflang={en-GB}
}

\begin{document}

\title{On Some Problems from the Kourovka Notebook}

\author{Wouter van Doorn}
\address{Groningen, the Netherlands}
\email{wonterman1@hotmail.com}

\author{Elias Judin}
\address{Department of Mathematics and Applied Mathematics, University of Cape
Town, Rondebosch 7701, South Africa}
\email{ejudin@gmail.com}

\author{Pietro Monticone}
\address{Harmonic, London, United Kingdom}
\email{pietro.monticone@harmonic.fun}

\author{Daniel Morrison}
\address{Orange, California, United States}
\email{morrison.danielj@gmail.com}


\begin{abstract}
The \emph{Kourovka Notebook} is a long-running collection of open problems in
group theory. In this paper we present solutions to eight of its problems. We
construct a group with exactly two maximal locally soluble normal
subgroups and show that, for every $1 \le k\le n!$, there is a group containing
$n$ distinct elements whose $n!$ ordered products take exactly $k$ distinct
values. We also give examples showing that group order together with the
statistic $\sum_g\varphi(\lvert g\rvert)$ does not determine simplicity, and we
construct a surjective non-injective Rota--Baxter operator on a non-abelian
group. Further, we determine the group generated by the class transpositions of
moduli at most $k$, prove that every power graph of a finite group that is a
cograph is chordal, show that the right-relatively convex subgroups of a
right-orderable group need not form a sublattice of its subgroup lattice, and
disprove a proposed rank inequality for certain $p$-group extensions. All of
these solutions were autonomously discovered and formally verified in Lean by
\textnormal{\href{https://aristotle.harmonic.fun/}{Aristotle}}, a formal
reasoning agent developed by
\textnormal{\href{https://www.harmonic.fun/}{Harmonic}}.
\end{abstract}

\maketitle

\section{Introduction}

The \emph{Kourovka Notebook} is a collection of open problems in group theory.
It was first proposed in 1965 at a group theory meeting in the small town of
Kourovka. The first issue of the Notebook was published later that same
year~\cite{Kourovka}. Since then, new issues containing both new problems and
solutions to earlier ones have appeared every two to four years; the
twenty-first issue appeared in early 2026. Today, the Notebook serves as a
standard reference for open problems in group theory and contains problems
proposed by mathematicians from all over the world. It has been highly
successful in promoting communication and collaboration among mathematicians and
has spurred solutions to hundreds of problems, including most of those from the
earliest issues.

We present solutions to eight problems from the Kourovka Notebook, including one
that first appeared in the third issue in 1969. The problems
cover a wide variety of topics in group theory, and their solutions take the
form of proofs, counterexamples, and constructions. What unites them is the
method by which they were discovered: each solution was developed and formalised
by Aristotle, a formal reasoning agent developed by Harmonic~\cite{Aristotle}.

Formalisation of mathematics is the process by which mathematical definitions,
statements, and proofs are encoded in the language of a proof assistant (in our
case Lean~4~\cite{Lean}), after which a computer can logically verify the work
step by step.
Formalisation is valuable for many reasons, not least because it allows the
rigorous verification of difficult or technical proofs in which human error is
particularly likely. However, the precise nature of the work makes formalisation
time-consuming and challenging, even when solutions are already known, partly
because proofs written for human readers often omit certain details or rely on
unstated assumptions that the reader is expected to infer. Aristotle aims to
automate formalisation and, as the results of this paper demonstrate, to apply
formal reasoning to the development of new solutions.

Therefore, the objective of this work is twofold. First, we present new
solutions to problems from the Kourovka Notebook in conventional mathematical
prose. Second, we describe the role of formalisation and
autonomous agents in the development of these solutions. The first objective is
addressed in the main sections of this paper, where we present natural language
solutions to each problem. We now summarise the eight results.

In Section~\ref{sec:346} we construct a group with exactly two maximal locally
soluble normal subgroups, resolving Problem~3.46. In Section~\ref{sec:1850} we
solve Problem~18.50 and show that, for all positive integers $k$ and $n$ with $1
\le k \le n!$, there exists a group $G$ containing $n$ distinct elements whose
$n!$ products take on exactly $k$ distinct values. In Section~\ref{sec:1925} we
answer Problem~19.25 in the negative by exhibiting groups $G$ and $H$ of equal
order and equal totient sum, with $G$ simple and $H$ not. In
Section~\ref{sec:20125} we solve Problem~20.125 by constructing a surjective but
non-injective Rota--Baxter operator on a non-abelian group. In
Section~\ref{sec:218} we answer Problem~21.8 affirmatively and prove that the
horizontal class transpositions with modulus at most $k$ generate a group which
is isomorphic to the symmetric group on $\lcm(2, 3, \dots, k)$ elements. In
Section~\ref{sec:2124} we prove that if the power graph of a finite group has an
induced cycle of length four, then it also has an induced path on four vertices.
Equivalently, we answer Problem~21.24 affirmatively; if the power graph of a
finite group is a cograph, then it is chordal. In Section~\ref{sec:21147} we
solve Problem~21.147 by showing that the right-relatively convex subgroups of a
right-orderable group need not form a sublattice of its subgroup lattice.
Finally, for Problem~21.150 we let $G$ be an extension of a normal elementary
abelian subgroup $A$ by an elementary abelian group $B\cong G/A$, and we suppose
that $A$ contains an element $a$ for which $C_B(a)$ is trivial. In
Section~\ref{sec:21150} we give an example of such a $p$-group $G$ for which the
inequality $\rank\bigl(Z(H)\cap H'\bigr)\le\rank(B)$ fails, where $H=\langle
a,B\rangle$.

These sections are independent and self-contained, and can be read without any
knowledge of formalisation. They also contain the necessary background, as well
as references to the relevant literature. In Appendix~\ref{sec:appendix} we
elaborate on the process of formal discovery and provide further details on our
methodology. The formal solutions we refer to are all available in our GitHub
repository.\footnote{\url{https://github.com/pitmonticone/Kourovka}}

We believe that the responsible use of automated formal reasoning tools requires
clear disclosure of how they were used and what role humans played in the
resulting work. In this case, the solutions were discovered autonomously by
Aristotle; by this we mean that it developed the proof strategies without human
guidance about which mathematical steps to take. The authors monitored Aristotle's
work, intervened when clarifications were needed and finally worked with
Aristotle to develop a library of relevant concepts. Thus, Aristotle was
responsible for developing the new mathematical arguments, while the authors
curated the project, provided the infrastructure, and
ultimately wrote this paper; both parts were necessary for the final result.

Once the initial formal solution was developed, we \emph{informalised} the
proof. That is, we translated the proof produced by Aristotle in Lean~4 into
conventional mathematical prose. We then shared both the formal and informal
proofs with the Kourovka Notebook editors, the problem proposers, and reviewers.
These initial submissions remain available in the official preprint repository
maintained on the
\href{https://kourovkanotebookorg.wordpress.com/repository/}{\emph{Kourovka Notebook}
website} and are cited in the corresponding sections. Once a solution was accepted,
we polished the informal exposition and canonised the formal codebase so that at
least part of it could be upstreamed to Mathlib~\cite{Mathlib}. We extracted and
generalised reusable results, refactored the code, and improved its organisation
and documentation. We did this work partly by hand and partly with Aristotle.
The informal proofs presented here and the formal proofs in the companion
repository are intended to complement one another: the former expose the
mathematical ideas, while the latter provide machine-verifiable accounts of the
arguments. Together, these results show that autonomous formal reasoning can
move beyond verification to mathematical discovery.

\section{Problem 3.46: Maximal locally soluble normal subgroups}
\label{sec:346}

We say that a group is \emph{locally soluble} if every finitely generated
subgroup is soluble. These groups have been extensively studied since the work
of Chernikov~\cite{Chernikov1940, Chernikov1943}. Their theory was developed
further by, among others, Roseblade~\cite{Roseblade1962} and
Wilson~\cite{Wilson1977}.
Relevant surveys with additional references are those by
Plotkin~\cite{Plotkin5861} and Robinson~\cite{Robinson1972}.

We define a normal subgroup of a group to be \emph{maximal locally soluble} if
it is locally soluble and not properly contained in any other locally soluble
normal subgroup. With this definition, we first observe that every group
has at least one such subgroup. Indeed, the trivial subgroup is certainly
locally soluble, and one can check that it must be contained in a maximal
element by Zorn's lemma.

There are a few cases in which it is not hard to prove that a maximal locally
soluble normal subgroup must also be unique. For example, any locally soluble
group (in particular every soluble group) has exactly one such subgroup, namely
itself. Somewhat less trivially, every finite group also has a unique such
subgroup. For finite groups, local solubility coincides with solubility, and the
product of any two soluble normal subgroups is soluble. Hence the join of all
soluble normal subgroups, known as the soluble radical, is the unique maximal
locally soluble normal subgroup. More generally, a group has exactly one
maximal locally soluble normal subgroup precisely when any product of two
locally soluble normal subgroups is again locally soluble.

Plotkin's survey on soluble and nilpotent groups~\cite{Plotkin5861} asked the
natural question whether it is true in general that the product of two locally
soluble normal subgroups is again locally soluble, which, if true, would thereby
imply the uniqueness of a maximal locally soluble normal subgroup. As it turns
out, the answer is no; Baumslag, Kov\'acs and Neumann~\cite{BKN1965} managed to
construct a finitely generated group $G$ that can be written as a product of two
locally soluble normal subgroups, despite the fact that $G$ itself is not
locally soluble. As their construction provides a group which has infinitely
many maximal elements, this prompted Plotkin to wonder whether there was a
dichotomy at play; is it the case that every group either has a unique maximal
locally soluble normal subgroup, or has infinitely many of them? The Kourovka
Notebook records the resulting question as Problem~3.46~\cite{Kourovka}: does
there exist a group having more than one but finitely many maximal locally
soluble normal subgroups? As in the Notebook submission~\cite{KNote346}, here we
answer these questions by constructing a group with exactly two maximal locally
soluble normal subgroups.

\begin{theorem}
There exists a group $G$ with exactly two maximal locally soluble normal
subgroups $M_1$ and $M_2$.
\end{theorem}

\begin{proof}
We consider infinite integer matrices whose rows and columns are indexed by the
lattice $\Z^2$, and we restrict to unipotent matrices with finitely many nonzero
off-diagonal entries. For such a matrix $U$, let $f\colon\Z^2\times\Z^2\to\Z$ be
the (finitely supported) function representing the off-diagonal part of $U$.
That is, we can write $U=I+f$, where, for distinct $p,q\in\Z^2$, $f(p,q)$ is the
$(p,q)$-entry of $U$. Since $U$ and $f$ determine one another, we freely switch
between these two viewpoints. For example, matrix multiplication corresponds to
\[
(f\ast g)(p,q):=\sum_{k\in\Z^2} f(p,k)\,g(k,q),
\]
which is a finite sum. With $f^{\ast1} := f$, we write $f^{\ast m} := f\ast
f^{\ast (m-1)}$ for the $m$-fold product. For the matrices we work with, we have
$(I+f)(I+g)=I+(f+g+f\ast g)$, so the induced group law on the corresponding
functions can be written as
\[
f\cdot g:=f+g+f\ast g.
\]
\begin{definition}
Writing a lattice point as
$p=(p^{(1)},p^{(2)})$, call both $U$ and its corresponding $f$
\emph{product-valid} if every $(p,q)\in\supp f$ satisfies $p^{(1)}<q^{(1)}$ and
$p^{(2)}<q^{(2)}$, and \emph{anti-valid} if every such $(p,q)$ satisfies
$p^{(1)}<q^{(1)}$ and $q^{(2)}<p^{(2)}$. Let $\UT_1$ and $\UT_2$ be the sets of
product-valid and anti-valid matrices respectively, and set
$K:=\UT_1\times\UT_2$.
\end{definition}

\begin{lemma}\label{lem:ut-group}
Both $\UT_1$ and $\UT_2$ are groups; hence so
is $K$. Moreover, if $I+f\in \UT_1\cup\UT_2$, then $f^{\ast m}=0$ for all $m >
\lvert \supp f \rvert$, and for the inverse of $I+f$ we have the equality
\begin{equation}\label{eq:utinverse}
(I+f)^{-1} = I-f+f^{\ast2}-f^{\ast3}+\cdots\pm f^{\ast \lvert \supp f \rvert}.
\end{equation}
\end{lemma}

\begin{proof}
Assume that $f$ and $g$ are product-valid and let $(p,q)$ be an element of
$\supp(f\cdot g)$. Then $(p,q)$ lies in $\supp f$, in $\supp g$, or in
$\supp(f\ast g)$; in each case a nonzero contribution forces $p^{(1)}<q^{(1)}$
and $p^{(2)}<q^{(2)}$ by transitivity along the intermediate index, so that
$\UT_1$ is closed under taking products. Further, if $f^{\ast m}(p,q) \neq 0$,
then there are points $p = p_0, p_1, \ldots, p_m = q$ with every pair $(p_{i-1},
p_i) \in \supp f$. Since $f$ is product-valid by assumption, this gives a
strictly increasing chain $p_0^{(1)}<p_1^{(1)}<\cdots<p_m^{(1)}$ of first
coordinates. In particular, $p_i \neq p_j$ for all $1 \le i < j \le m$, which
implies $m \le \lvert \supp f \rvert$. Hence, multiplying the right-hand side of
equation~\eqref{eq:utinverse} by $I+f$ telescopes to $I$, giving the required
formula for the inverse of $I+f$. Since finite sums of product-valid functions
are again product-valid, we deduce that $\UT_1$ is closed under taking inverses.
As associativity follows from the general associativity of matrix
multiplication, we conclude that $\UT_1$ is indeed a group. The same arguments
apply to the anti-valid case $\UT_2$, which finishes the proof.
\end{proof}

For $(a,b) \in \Z^2$, let $\varphi_{a,b}$ be the automorphism of $K$ induced by
the lattice translation $(x,y) \mapsto (x+a,\,y+b)$ applied to both the row and
column indices of every off-diagonal entry. It preserves both validity
conditions, is multiplicative, and satisfies $\varphi_{0,0}=\mathrm{id}$ and
$\varphi_{a_1,b_1}\circ\varphi_{a_2,b_2}=\varphi_{a_1+a_2,\,b_1+b_2}$. We then
define the group $G$ and its subgroups $M_1$ and $M_2$ as follows.

\begin{definition}
Let $G$ be the semidirect product $G := K\rtimes\Z^2$, with
\[
(k_1,a_1,b_1)(k_2,a_2,b_2) := \bigl(k_1\,\varphi_{a_1,b_1}(k_2),\;a_1+a_2,\;b_1+b_2\bigr).
\]
For $g=(k,a,b) \in G$, call $a$ and $b$ its \emph{translation coordinates}, and
set
\[
M_1 := \{g\in G: b=0\} \qquad \text{and} \qquad M_2 := \{g\in G: a=0\}.
\]
\end{definition}

As for $M_1$ and $M_2$, it is clear that neither is contained in the other.
Furthermore, as they are the kernels of the homomorphisms that send $g = (k, a,
b)$ to $b$ and $a$ respectively, they are indeed normal. It therefore suffices
to show that they are locally soluble, and that every other locally soluble
normal subgroup of $G$ is contained in one of them. We begin with local
solubility. Let $H \le M_1$ be a finitely generated subgroup; the case $H \le
M_2$ is analogous. We must prove that $H$ is soluble.

The homomorphism that sends $(k, a, 0) \in M_1$ to $a \in \Z$ has kernel
(isomorphic to) $K$. Restricting the homomorphism to $H$, we see that $H/(H\cap
K)$ embeds in $\Z$ by the first isomorphism theorem, and is therefore abelian.
Since an extension of a soluble group by an abelian group is again soluble, in
order to show that $H$ is soluble, it suffices to prove that $J := H\cap K$ is
soluble.

For each of the finitely many generators $(k, a, 0)$ of $H$, consider the second
coordinates of all row and column indices appearing in the supports of the two
matrix components of $k$. Since each $k$ has finite support, these coordinates
form a finite set; let $S\subseteq \Z$ be the union of these finite sets over
all generators and assume without loss of generality that $S$ is nonempty. We
claim that every element of $H$ has its $K$-component supported only on row and
column indices whose second coordinates lie in $S$. Indeed, in the semidirect
product multiplication inside $M_1$, the only translations that occur preserve
second coordinates, while the convolution product and inverse formula introduce
no new second-coordinate values. Hence all second coordinates occurring in the
$K$-components of elements of $H$ already occur among the generators.

For a positive integer $d$, let $U_d$ be the subgroup of $\UT_1$ consisting of
those elements $I+f$ for which every $(p,q) \in \supp f$ satisfies
$q^{(2)}-p^{(2)}\ge d$. To see why these elements form a subgroup, we more
generally note that, if $I+f\in U_d$ and $I+g\in U_e$, then every nonzero
contribution to $f \ast g$ comes from some $p,r,q$ with $f(p,r)g(r,q) \neq 0$,
giving $q^{(2)}-p^{(2)} \ge d+e$. Hence $U_d$ is closed under taking products,
while the inverse formula of Lemma~\ref{lem:ut-group} shows that $U_d$ is closed
under taking inverses as well. Moreover, in the commutator expansion of two
elements of $U_d$ and $U_e$, all linear terms cancel by the formula for the
inverse, and every remaining term contains at least one factor from each side.
In particular, $[U_d,U_e] \le U_{d+e}$.

Now let $L \le \UT_1$ be a subgroup whose entries use only second coordinates
from $S$ and set $D := \max S-\min S$. Every nonzero entry of an element of $L$
has a vertical gap between $1$ and $D$. By the commutator estimate, induction
gives $L^{(i)} \le U_{2^i}$, where $L^{(i)}$ is the $i$-th derived subgroup.
Choose $i$ with $2^i>D$. Then no nonzero entry remains in $L^{(i)}$, so
$L^{(i)}=1$ and $L$ is soluble. The same argument applies to subgroups of
$\UT_2$, using the gap $p^{(2)}-q^{(2)}$ instead.

To finish, let $\pi_1, \pi_2$ be the projections from $K = \UT_1\times\UT_2$
to its two factors. By applying the previous paragraph to $L := \pi_1(J)$ and
$L:= \pi_2(J)$ we find that both projections are soluble. Since $J\le
\pi_1(J)\times\pi_2(J)$, we finally deduce that $J$ is soluble as well, proving
the local solubility of $M_1$.

What remains to be shown is that for every locally soluble normal subgroup $N
\le G$ we have either $N \le M_1$ or $N \le M_2$. Contrapositively, let $N$ be a
normal subgroup contained in neither $M_1$ nor $M_2$. We will then show that
there exists a finitely generated subgroup $H \le N$ such that $H$ is not
soluble.

Since $N \not \le M_1$ and $N \not \le M_2$, $N$ must contain an element $g_1$
with nonzero first translation coordinate and an element $g_2$ with nonzero
second translation coordinate. One can then check that one of
$g_1,g_1g_2,g_1^2g_2$ must have
both translation coordinates nonzero and, after replacing it by its inverse if
necessary, we obtain an element $g = (k,a,b) \in N$ with $a > 0$ and $b \neq 0$.

Let $m_0 \ge 0$ be the largest absolute value of any coordinate that appears in
a row or column index of a nonzero entry of $k$ or $k^{-1}$. Moreover, for
integers $m$ and $n$ with $m_0 \le m < n$, define $e_{m,n}$ to be the unipotent
matrix with the single nonzero off-diagonal entry $1$ at the position
$\bigl((ma,mb),(na,nb)\bigr)$. If $b>0$, then $ma<na$ and $mb<nb$, so
$e_{m,n}\in\UT_1$. If $b<0$, then $nb<mb$, so $e_{m,n}\in\UT_2$. Either way, we
have the following two identities, where the commutator $[x,y]$ is defined as
$xyx^{-1}y^{-1}$.

\begin{lemma}\label{lem:single}
For all integers $\ell, m, n$ with $m_0 \le m <
\ell < n$ we have
\[
\bigl[e_{m,\ell},e_{\ell,n}\bigr] = e_{m,n}
\qquad \text{and} \qquad
ge_{m,n}g^{-1} = e_{m+1,n+1},
\]
where for the second identity we identify $e_{m,n}$ with the corresponding
element of $G$ which has all other components equal to the identity.
\end{lemma}

\begin{proof}
Using $e_{m,n} = I+E_{m,n}$, the first identity follows by direct computation
of the commutator
\[
\bigl[e_{m,\ell},e_{\ell,n}\bigr]
= (I+E_{m,\ell})(I+E_{\ell,n})(I-E_{m,\ell})(I-E_{\ell,n}),
\]
as all terms cancel or vanish except $I$ and
$E_{m,\ell}E_{\ell,n}=E_{m,n}$, leaving $I+E_{m,n}=e_{m,n}$. As for the second
identity, $ge_{m,n}g^{-1}$ simplifies to
$k\varphi_{a,b}(e_{m,n})k^{-1} = ke_{m+1,n+1}k^{-1}$ in the $K$-component, with
translation coordinates equal to $0$. Since $m_0 < m+1 < n+1$, none of the row
or column indices of $e_{m+1,n+1}$ appear in the support of $k$ or $k^{-1}$.
Hence, multiplication by $k$ on the left and by $k^{-1}$ on the right does not
change this entry, while the remaining factors cancel. This finishes the proof.
\end{proof}

Since $N$ is normal, the commutator $[u,g]=(ugu^{-1})g^{-1}$ lies in $N$
for every $u \in G$. Repeating this, we find that the iterated commutator
$[u, [u,g]]$ also lies in $N$ for every $u$. With $u_i :=
e_{m_0,m_0+2^i}$ we will use this fact with $u_0$, and we further define $v_i :=
e_{m_0+2^i,m_0+2^{i+1}}$ for $i \ge 0$. By applying Lemma~\ref{lem:single}, we
then find
\begin{align*}
[u_0, [u_0,g]]&= \bigl[u_0, u_0v_0^{-1}\bigr] \\
&= u_0^2v_0^{-1}u_0^{-1}v_0u_0^{-1} \\
&= (I+E_{m_0,m_0+1})^2 (I-E_{m_0+1,m_0+2})
  (I-E_{m_0,m_0+1}) (I+E_{m_0+1,m_0+2}) (I-E_{m_0,m_0+1}) \\
&= I-E_{m_0,m_0+2} \\
&= u_1^{-1} \in N.
\end{align*}

We now define $H$ as the subgroup of $N$ generated by $g$ and $u_1$, and aim to
show that $H$ is not soluble. More precisely, we claim that for every $i \ge 1$,
the derived subgroup $H^{(i-1)}$ contains $u_{i}$, which is certainly true for
$i = 1$. By induction, Lemma~\ref{lem:single} gives $g^{2^{i}} u_i g^{-2^{i}} =
v_{i} \in H^{(i-1)}$. Applying Lemma~\ref{lem:single} once more, we find that
$u_{i+1} = \bigl[u_{i}, v_{i} \bigr] \in H^{(i)}$, finishing the proof.
\end{proof}

\section{Problem 18.50: Permuted product sets of prescribed cardinality}
\label{sec:1850}

Consider a group $G$ with distinct elements $g_1,\dots,g_n$. Then there are $n!$
ways to arrange a product of those $n$ elements, although different products may
result in the same group element. In the case of an abelian group the order is
irrelevant and all $n!$ products are equal to $g_1\cdots g_n$,
while rewritable groups are characterised by the property that at least two
permutations yield the same group element. Problem~18.50 of the Kourovka
Notebook~\cite{Kourovka}, proposed by S.~Kohl, asks whether, for every
$n\in\N$ and $k\in\{1,\dots,n!\}$, there is a group $G$ with $n$ distinct
elements whose $n!$ permuted products take exactly $k$ values. Kohl had first
raised the question on MathOverflow in March 2013~\cite{MO1850}, where he
confirmed it computationally for small $n$ and later announced that it would
appear in the eighteenth issue of the Notebook. Thirteen years on, he has
recorded the construction below as an answer on his own thread. More precisely,
we prove the following theorem using the
construction originally submitted to the Kourovka Notebook~\cite{KNote1850}:

\begin{theorem}
For every $n\in\N$ and every $k \in \{1,\dots,n!\}$, there exists a group
$G$ and pairwise distinct elements $g_1,\dots,g_n\in G$ such that the set
\[
\{g_{\sigma(1)}\cdots g_{\sigma(n)} : \sigma\in S_n\}
\]
of all permuted products of the $g_i$ has cardinality $k$.
\end{theorem}

The key idea of the construction is to work in a central extension where the
commutators of the elements $g_i$ record an inversion code uniquely identifying
the permutation $\sigma \in S_n$. We then pass to a finite central quotient
that identifies the inversion codes modulo $k$, thereby forcing exactly $k$
possible products.

\begin{proof}
Fix positive integers $k$ and $n$ with $k \le n!$. For
$\pi \in S_n$, define
\[
a_j(\pi) := \#\{\,i<j : \pi^{-1}(j) < \pi^{-1}(i)\},
\]
so that $a_j(\pi)$ counts the inversions of $\pi$ whose larger entry is $j$.
We have $0 \le a_j(\pi) \le j-1$ since $i < j$, while the sequence
$(a_1(\pi),\dots,a_n(\pi))$ is the usual inversion vector of $\pi$. Furthermore,
as $\pi$ ranges over $S_n$, the values $a_j(\pi)$ range independently through
the integers from $0$ to $j-1$. Define
$R(\pi):=\sum_{j=1}^n a_j(\pi)(j-1)!$. The map $R$ is a bijection between $S_n$
and the set $\{0,1,\dots,n!-1\}$.

Let $G := \Z^n\times\Z/k\Z$. For $u=(u_1,\dots,u_n)$ and $v=(v_1,\dots,v_n)$ in
$\Z^n$, set
\[
B(u,v) := \sum_{1\le i<j\le n}u_jv_i(j-1)! \in \Z/k\Z,
\]
and define a multiplication on $G$ by
\[
(u,r)(v,s) := \bigl(u+v,\,r+s+B(u,v)\bigr).
\]
The identity element is $(0,0)$, and inverses are given by
$(u,r)^{-1}=(-u,\,-r+B(u,u))$. The bilinearity of $B$ guarantees that this
operation is associative, so $G$ is a group.

Let $e_1,\dots,e_n$ be the standard basis of $\Z^n$, and set $g_j := (e_j,0)$
for each $1 \le j \le n$. If $i<j$, then $g_ig_j = (e_i+e_j,0)$, whereas $g_jg_i
= (e_i+e_j,(j-1)!) = g_ig_j(0,(j-1)!)$. The last factor is central because it
lies entirely in the $\Z/k\Z$-coordinate. Thus, relative to the ordered
pair $g_ig_j$, an inverted pair $g_jg_i$ contributes $(j-1)!$ to the central
coordinate.

Take an arbitrary $\pi\in S_n$. Sorting the word $g_{\pi(1)}g_{\pi(2)}\cdots
g_{\pi(n)}$ into the increasing word $g_1g_2\cdots g_n$ requires exactly
$a_j(\pi)$ such swaps involving the letter $j$. Therefore
\[
g_{\pi(1)}g_{\pi(2)}\cdots g_{\pi(n)}
=\bigl((1,\dots,1),\ \textstyle\sum_{j=1}^n a_j(\pi)(j-1)!\bigr)
=\bigl((1,\dots,1),\,R(\pi)\bmod k\bigr).
\]
All permuted products share the $\Z^n$-coordinate $(1,\dots,1)$, so that the
only variation is in the central coordinate. Since $R$ is a bijection between
$S_n$ and $\{0,1,\dots,n!-1\}$ and $k \le n!$, the residues $R(\pi) \bmod k$ run
through every element of $\Z/k\Z$. They cannot give more than $k$ values, and
they give at least $k$ values because $0,1,\dots,k-1$ all occur among the values
of $R$. Hence, the set of permuted products has cardinality exactly $k$.
\end{proof}

\section{Problem 19.25: Orders, totient sums, and simplicity}
\label{sec:1925}

Problem~12.39 of the Kourovka Notebook~\cite{Kourovka} asked whether a finite
group with the same order and spectrum $\omega(X) := \{\lvert x\rvert:x\in X\}$
as a finite simple group must be isomorphic to the simple group, and was
answered in the positive by Vasil'ev, Grechkoseeva, and
Mazurov~\cite{VasilevGrechkoseevaMazurov}. One can then ask whether weaker
information than the spectrum could provide similar results. One possibility
(for a finite group $X$) concerns the sum
\[
\sum_{x \in X} \varphi(\lvert x\rvert) = \sum_{m \ge 1} e_m(X)\varphi(m),
\]
where $\varphi$ is Euler's totient function and $e_m(X) := \lvert\{x\in X:\lvert
x\rvert=m\}\rvert$ is the number of elements of order $m$ in $X$. Curtin and
Pourgholi~\cite{CurtinPourgholi} studied this totient sum and observed its
connection with bidirectional edges of the directed power graph: if
$\overleftrightarrow{E}(X)$ denotes the set of bidirectional unordered pairs of
distinct vertices, then $\lvert\overleftrightarrow{E}(X)\rvert =
\tfrac12\bigl(\sum_{x\in X}\varphi(\lvert x\rvert)-\lvert X\rvert\bigr)$.
Problem~19.25, proposed by B.~Curtin and G.~R.~Pourgholi, asks whether every
group $H$ with the same order and totient sum as a simple group $G$ must also be
simple. We show that this is false by constructing such groups with $G$ simple
and $H$ not.

\begin{theorem}\label{thm:1925}
There are finite groups $G$ and $H$ of order 6048 with
\[
\sum_{g\in G}\varphi(\lvert g\rvert) = \sum_{h\in H}\varphi(\lvert h\rvert) = 23984,
\]
such that $G$ is simple and $H$ is not.
\end{theorem}

The proof of Theorem~\ref{thm:1925} is split into two lemmas: one constructing
$G$ and one constructing $H$.

\begin{lemma}\label{lem:1925sim}
The projective special unitary group $G:=\PSU(3,3)$ is simple of order $6048$
and has totient sum $23984$.
\end{lemma}

Recall the definition of $\PSU(3,3)$. Consider a $3$-dimensional vector space
$V$ over $\F_9$. The field $\F_9$ has a nontrivial automorphism $a\mapsto a^3$,
which plays the role of complex conjugation, and we equip $V$ with a
nondegenerate Hermitian form $f$ with respect to this automorphism. Among the
one-dimensional subspaces of $V$, exactly $28$ are
isotropic~\cite[Lemma~10.4]{TaylorClassicalGroups}, meaning that they are of the
form $\langle v \rangle$ for some nonzero $v \in V$ satisfying $f(v,v) = 0$. We
denote these $28$ isotropic one-dimensional subspaces by $\Omega$. The group
$\SU(3,3)$ is then defined as the group of linear automorphisms $L$ of $V$ that preserve
$f$ and have determinant $1$. Every such $L$ permutes the elements of $\Omega$,
because if $f(v, v) = 0$ for some $v \neq 0$, then $L(v) \neq 0$ has the
property that $f(L(v), L(v)) = f(v, v) = 0$. The group $\PSU(3,3)$ is now
defined as the central quotient $\SU(3,3)/Z(\SU(3,3))$, and this quotient still
acts on $\Omega$.

The original solution note~\cite{KNote1925} proved the simplicity of $\PSU(3,3)$
using a concrete permutation representation as a subgroup of
$S_{28}$~\cite{AtlasU33}, together with exhaustive computations involving
conjugacy classes and normal closures.
However, a cleaner and less computational proof uses the natural action of
$\PSU(3,3)$ on its $28$ isotropic points to prove
simplicity by Iwasawa's criterion~\cite{Iwasawa1941}; a perfect finite group $G$
is simple if it acts faithfully and primitively on a set $X$ in such a way that,
for some $x \in X$, the stabiliser $G_x$ of $x$ contains an abelian normal
subgroup $A$ such that the conjugates of $A$ generate $G$. A proof along these
lines can be found in~\cite[Theorem~10.15]{TaylorClassicalGroups}, although we
give a brief overview here.

\begin{proof}
Since $V$ is a nondegenerate $3$-dimensional space containing isotropic points,
it has Witt index $1$. It therefore follows
from~\cite[Theorem~10.12(i)]{TaylorClassicalGroups} that the natural action of
$G$ on $\Omega$ is faithful and doubly transitive, and hence primitive.

Fix a point $P := \langle u\rangle\in\Omega$ and, following Taylor's convention
that $f$ is linear in its first argument, define $t_a(v) := v+a f(v,u)u$ for
each $a\in\F_9$ satisfying $a+a^3=0$. These transformations are unitary
transvections: they preserve $f$, have determinant $1$, and fix the orthogonal
complement $\{v\in V:f(v,u)=0\}$ pointwise. The images in $G$ of the
transformations $t_a$, as $a$ ranges over all the elements satisfying $a+a^3=0$,
form an abelian subgroup $A$. Moreover, conjugation by an element of the
stabiliser $G_P$ sends a transvection with centre $P$ to another such
transvection. Hence $A$ is normal in $G_P$.

The conjugates of $A$ are the analogous transvection subgroups associated with
the other points of $\Omega$. By Lemmas~10.13--10.14 and the proof of
Theorem~10.15 in Taylor~\cite{TaylorClassicalGroups}, the unitary transvections
generate $\SU(3,3)$ and belong to its commutator subgroup. Consequently, their
images generate $G$ and belong to its commutator subgroup. It therefore follows
that $G$ is perfect and Iwasawa's criterion now shows that $G$ is simple. For
the complete proof, we refer to Taylor~\cite{TaylorClassicalGroups}.

The general formula for the order of projective special unitary
groups~\cite[p.~118]{TaylorClassicalGroups} is
\[
  \lvert\PSU(n,q)\rvert
  =
  \frac{q^{n(n-1)/2}}{\gcd(n,q+1)}
  \prod_{i=2}^{n}\bigl(q^i-(-1)^i\bigr).
\]
Specialising to $n = q = 3$ gives
\[
  \lvert\PSU(3,3)\rvert
  =
  \frac{3^3(3^2-1)(3^3+1)}{\gcd(3,4)}
  =
  6048.
\]
The finite element-order distribution used in the Lean verification is
\[
\begin{array}{c|cccccccc}
m&1&2&3&4&6&7&8&12\\ \hline
e_m(G)&1&63&728&504&504&1728&1512&1008
\end{array}.
\]
Therefore
\[
\sum_{g\in G}\varphi(\lvert g\rvert) =
1+63+728\cdot2+504\cdot2+504\cdot2+1728\cdot6+1512\cdot4+1008\cdot4
= 23984. \qedhere
\]
\end{proof}

\begin{lemma}\label{lem:1925nsim}
Let $F := C_7 \rtimes C_6$ be the semidirect
product in which a generator of $C_6$ acts on $C_7$ by multiplication by $3$
modulo $7$, and $H := C_6 \times S_4 \times F$. Then $H$ has order $6048$ and
totient sum $23984$, but $H$ is not simple.
\end{lemma}

\begin{proof}
We immediately have $\lvert H\rvert = 6 \cdot 24 \cdot 42 = 6048$, while $H$ is
not simple as $C_6 \times \{1\} \times \{1\}$ is a proper nontrivial normal
subgroup. What remains is to compute $\sum_{h\in H}\varphi(\lvert h\rvert)$. The
factors $C_6$ and $S_4$ have order distributions
\[
\begin{array}{c|rrrr}
 m&1&2&3&6\\ \hline
 e_m(C_6)&1&1&2&2
\end{array}
\qquad
\begin{array}{c|rrrr}
 m&1&2&3&4\\ \hline
 e_m(S_4)&1&9&8&6
\end{array}.
\]
As for $F$, we identify $C_7$ with the additive group of $\F_7$. The normal
subgroup $C_7$ contributes one element of order $1$ and six elements of order
$7$. Let $j \neq 0$ have order $d$ in $C_6$. Then the scalar $3^j$ acting on
$\F_7$ also has order $d$, since $3$ has order $6$ in $(\Z/7\Z)^{\times}$. So
for any $v\in\F_7$,
\[
(v,j)^d = \bigl((1+3^j+\cdots+3^{(d-1)j})v,0\bigr) = (0,0),
\]
because $3^j \neq 1$ and $(3^j)^d = 1$. Then the order of $(v,j)$ is also $d$
regardless of $v$. Therefore the order distribution of $F$ is given by
\[
\begin{array}{c|rrrrr}
 m&1&2&3&6&7\\ \hline
 e_m(F)&1&7&14&14&6
\end{array}.
\]
In a direct product, the order of an element is the least common multiple of the
orders of its components, so the order distribution of $H$ is
\[
\begin{array}{c|rrrrrrrrrrrr}
 m&1&2&3&4&6&7&12&14&21&28&42&84\\ \hline
 e_m(H)&1&159&404&96&3324&6&1200&114&156&72&372&144
\end{array}.
\]
Then the totient sum is
\[
\begin{aligned}
\sum_{h\in H}\varphi(\lvert h\rvert)
&=1+159+404\cdot2+96\cdot2+3324\cdot2+6\cdot6\\
&\quad+1200\cdot4+114\cdot6+156\cdot12+72\cdot12+372\cdot12+144\cdot24\\
&=23984,
\end{aligned}
\]
as desired.
\end{proof}

Thus, by Lemmas~\ref{lem:1925sim} and~\ref{lem:1925nsim}, we have groups $G$ and
$H$ with $\lvert G\rvert=\lvert H\rvert=6048$ and $\sum_{g\in G}\varphi(\lvert
g\rvert)=\sum_{h\in H}\varphi(\lvert h\rvert)=23984$, with $G$ simple and $H$
non-simple, which completes the proof of Theorem~\ref{thm:1925}.

\begin{remark}
The group $H$ can be written in a form involving a wreath product. In the
following,
\[
  C_2\wr S_3=(C_2)^3\rtimes S_3
\]
denotes the imprimitive wreath product for the natural action of $S_3$ on three
letters. Regard $(C_2)^3$ as the vector space $\F_2^3$. The subspace
\[
  \Delta:=\langle(1,1,1)\rangle
\]
is fixed pointwise by $S_3$ and the subspace
\[
  U:=\{(u_1,u_2,u_3):u_1+u_2+u_3=0\}
\]
is $S_3$-invariant. Since
\[
  \F_2^3 = \Delta \oplus U,
\]
we have
\[
  (C_2)^3 \rtimes S_3 \cong \Delta \times (U \rtimes S_3).
\]
The group $U \rtimes S_3$ is the affine general linear group
$\operatorname{AGL}(2,2)$ and therefore isomorphic to $S_4$. Therefore
\[
  C_2\wr S_3\cong C_2\times S_4.
\]
It follows that
\[
  H\cong C_3\times(C_2\wr S_3)\times(C_7\rtimes C_6).
\]
This merely rewrites the same non-simple group, since
$C_3\times C_2\cong C_6$.
\end{remark}

\begin{remark}
The non-simple partner to $G$ is not unique. Let
\[
  W:=C_6\wr S_2=(C_6\times C_6)\rtimes S_2,
\]
where $S_2$ interchanges the two $C_6$-factors, and put
\[
  H':=C_2\times W\times F.
\]
Then $\lvert H'\rvert=2\cdot72\cdot42=6048$. The order distribution of $W$ is
\[
\begin{array}{c|rrrrrr}
 m&1&2&3&4&6&12\\ \hline
 e_m(W)&1&9&8&6&36&12
\end{array}.
\]
Indeed, the base subgroup $(C_6)^2$ contributes $1,3,8,24$ elements of orders
$1,2,3,6$, respectively. In the nontrivial coset, an element $(u,v)\tau$
satisfies
\[
  ((u,v)\tau)^2=(uv,uv),
\]
so its order is $2\lvert uv\rvert$. This gives $6,6,12,12$ elements of orders
$2,4,6,12$, respectively, in that coset. Hence the order distribution of $H'$
is
\[
\begin{array}{c|rrrrrrrrrrrr}
 m&1&2&3&4&6&7&12&14&21&28&42&84\\ \hline
 e_m(H')&1&159&134&96&3594&6&1200&114&48&72&480&144
\end{array}.
\]
Again, the entries sum to $6048$, and
\[
  \sum_{h'\in H'}\varphi(\lvert h'\rvert)=23984.
\]
Thus $H'$ is another non-simple group with the same order and the same totient
sum as $G$. Moreover, $H$ and $H'$ are not isomorphic, since
\[
  e_3(H)=404,
  \qquad
  e_3(H')=134.
\]
\end{remark}

\section{Problem 20.125: A surjective, non-injective Rota--Baxter operator}
\label{sec:20125}

A \emph{Rota--Baxter operator} (of weight $1$) on a group $G$ is a map $B\colon
G\to G$ satisfying
\[
B(g)\,B(h)=B\bigl(g\cdot B(g)\,h\,B(g)^{-1}\bigr)
\qquad(g,h\in G).
\]
Rota--Baxter operators were first proposed by G.~Baxter~\cite{Baxter} in the
context of commutative algebras, and only recently did Guo, Lang and
Sheng~\cite{GuoLangSheng} introduce the corresponding notion for (Lie) groups.
Bardakov and Gubarev~\cite{BardakovGubarev} extended this study to arbitrary
groups, providing many general constructions of such Rota--Baxter operators. For
example, they showed that any homomorphism from a group to one of its abelian
subgroups is a Rota--Baxter operator. In that paper, they also ask the following
question, which is now recorded as Problem~20.125 of the Kourovka
Notebook~\cite{Kourovka}: does there exist a non-abelian group $G$ and a
Rota--Baxter operator $B\colon G\to G$ such that $B$ is surjective but not
injective? As first explained in the submission to the Kourovka
Notebook~\cite{KNote20125}, the answer is yes; such a pair $(G, B)$ exists.

\begin{theorem}
There exists a non-abelian group $G$ and a Rota--Baxter operator $B\colon G\to
G$ that is surjective but not injective.
\end{theorem}

\begin{proof}
Let $G := S_3 \times \Z^{\N}$, where $S_3$ is the symmetric group on three
elements and $\Z^{\N}$ is the set of functions from $\N$ to $\Z$. On the latter
factor, group multiplication (which we will write additively) is defined as
function addition: $(g+h)(n) := g(n) + h(n)$. We then define $B \colon G \to G$
componentwise by $B(\sigma, f) := \bigl(\sigma^{-1},\; s(f)\bigr)$, where $s
\colon \Z^{\N} \to \Z^{\N}$ is the left-shift operator $s(f)(n) := f(n+1)$.

Since $B = (B_1, B_2)$ is defined componentwise, it suffices to verify the
Rota--Baxter equation separately for each factor. On the first factor, the
equation simplifies to the identity $g^{-1}h^{-1} = (hg)^{-1}$. On the abelian
factor $\Z^{\N}$, the right-hand side reduces to $s(g + h)$, which equals the
left-hand side because $s$ is a homomorphism.

Since $B_1$ is bijective, it suffices to check that $B_2 = s$ is surjective and
non-injective, but both of these are straightforward. Indeed, for any $g \in
\Z^{\N}$ one can define $h$ by $h(1) := 0$, and $h(n) := g(n-1)$ for all $n \ge
2$. Then $s(h) = g$, proving surjectivity. For non-injectivity, we can
define $h$ by $h(1) := g(1) + 1$ and $h(n) := g(n)$ for all $n \ge 2$ to get
$s(g) = s(h)$ with $g \neq h$. Finally, $G$ is non-abelian because $S_3$ is.
The same construction works with $S_3$ replaced by any non-abelian group.
\end{proof}

\section{Problem 21.8: Groups generated by horizontal class transpositions}
\label{sec:218}

Let $r(m) := \{r + tm \mid t \in \Z\}$ denote the residue class of $r$ modulo
$m$, where we assume $0 \le r < m$. Then, for each pair of disjoint residue
classes $r_1(m_1)$ and $r_2(m_2)$, we define the \emph{class transposition}
$\tau_{r_1(m_1), r_2(m_2)}$ as the mapping that exchanges $r_1 + t m_1$ with
$r_2 + t m_2$ for each integer $t$, while it fixes everything else. These class
transpositions generate a subgroup, denoted by $\CT(\Z)$, of the permutation
group of $\Z$. This subgroup was introduced by Kohl~\cite{KohlSimple}, who
established that $\CT(\Z)$ has a number of interesting properties, which include
being a countably infinite simple group with an uncountable collection of simple
subgroups and admitting embeddings of a variety of other groups. He later showed
that this group has connections to the Collatz conjecture as
well~\cite{KohlCollatz}.

Problem~21.8 of the Kourovka Notebook~\cite{Kourovka}, proposed by
V.~G.~Bardakov and A.~L.~Iskra, concerns class transpositions with the same
modulus, sometimes called \emph{horizontal class transpositions}. Such class
transpositions also generate a subgroup, which is denoted by $\CT_k$. Since
$\CT_k$ permutes the residue classes modulo $k$, it is isomorphic to the
symmetric group $S_k$. We then define $\CT_{(k)} := \langle \CT_2, \CT_3,
\ldots, \CT_k \rangle$ as the subgroup generated by horizontal class
transpositions with modulus at most $k$. Bardakov and Iskra~\cite{Bardakov}
developed the theory of horizontal class transpositions and used the computer
algebra system GAP to show that, for small $k \ge 4$ and with
$N$ the least common multiple of $2, 3, \ldots, k$, $\CT_{(k)}$ is isomorphic to
the symmetric group $S_N$. They then conjectured that this pattern continues
and, using an improved version of the original submission~\cite{KNote218}, we
prove that their conjecture is true for all $k \ge 4$. Pan~\cite{Pan}
recently proved the same result independently, using multiple transitivity.

\begin{theorem}\label{thm:218}
Let $L_k := \lcm(2, 3, \dots, k)$. For all $k \ge 4$, the group $\CT_{(k)}$ is
isomorphic to the symmetric group $S_{L_k}$.
\end{theorem}

\begin{proof}
We first explain how $\CT_{(k)}$ embeds in $S_{L_k}$. Every class transposition
$\tau_{r_1(d), r_2(d)}$ with $2 \le d \le k$ is periodic with period $d \mid
L_k$, so it is also $L_k$-periodic. Therefore each class transposition restricts
to a permutation of $\{0, 1, \dots, L_k - 1\}$, which induces a homomorphism
$\rho \colon \CT_{(k)} \to S_{L_k}$ defined by restricting each class
transposition $\tau_{r_1(d), r_2(d)}$ to $\{0, 1, \dots, L_k - 1\}$, equivalent
to $L_k / d$ periods of size $d$. The periodic behaviour guarantees this map is
injective, so $\CT_{(k)}$ is isomorphic to a subgroup of $S_{L_k}$.

Specifically, $\tau_{r_1(d), r_2(d)}$ maps to the permutation $\sigma_{r_1,
r_2}^{(d)}$, which acts periodically with period $d$ and whose action on $\{0,
1, \dots, d-1\}$ swaps $r_1$ and $r_2$. More concretely, $\sigma_{r_1,
r_2}^{(d)}$ is defined by the equation
\[
\sigma_{r_1, r_2}^{(d)}(n) := \bigl\lfloor n / d \bigr\rfloor \cdot d
  + \tau_{r_1, r_2}\bigl(n \bmod d\bigr),
  \qquad n \in \{0, \dots, L_k - 1\},
\]
where $\tau_{r_1, r_2}$ is the transposition of $r_1$ and $r_2$ in $\{0, \dots,
d - 1\}$.

For the other direction of Theorem~\ref{thm:218}, it suffices to show that the
image $\rho(\CT_{(k)})$ is all of $S_{L_k}$, which we prove by induction on $k$.
When $k = 4$, $L_k = 12$ and we show $\CT_{(4)} = S_{12}$ by constructing all
adjacent transpositions.

A direct computation shows that $\bigl(\sigma_{0,1}^{(3)}
\sigma_{0,1}^{(4)}\bigr)^3 = (6\;7)$. From $(6\;7)$, every adjacent
transposition is obtained by conjugation with class transposition generators:
\begin{align*}
(0\;1)   &= \left(\sigma_{0,3}^{(4)} \sigma_{0,1}^{(3)}\right)^2
  \sigma_{0,3}^{(4)} \;(6\;7)\; \sigma_{0,3}^{(4)}
  \left(\sigma_{0,1}^{(3)} \sigma_{0,3}^{(4)}\right)^2, \\
(1\;2)   &= \sigma_{2,3}^{(4)}
  \left(\sigma_{0,1}^{(3)} \sigma_{0,3}^{(4)}\right)^2 (6\;7)\;
  \left(\sigma_{0,3}^{(4)} \sigma_{0,1}^{(3)}\right)^2 \sigma_{2,3}^{(4)}, \\
(2\;3)   &= \sigma_{0,1}^{(3)} \sigma_{0,3}^{(4)} \sigma_{0,1}^{(2)}
  \sigma_{0,1}^{(3)} \sigma_{0,2}^{(4)} \;(6\;7)\;
  \sigma_{0,2}^{(4)} \sigma_{0,1}^{(3)} \sigma_{0,1}^{(2)}
  \sigma_{0,3}^{(4)} \sigma_{0,1}^{(3)}, \\
(3\;4)   &= \sigma_{0,2}^{(4)} \sigma_{0,1}^{(3)} \sigma_{0,2}^{(4)}
  \;(6\;7)\; \sigma_{0,2}^{(4)} \sigma_{0,1}^{(3)} \sigma_{0,2}^{(4)}, \\
(4\;5)   &= \sigma_{0,2}^{(4)} \sigma_{1,3}^{(4)}
  \;(6\;7)\; \sigma_{1,3}^{(4)} \sigma_{0,2}^{(4)}, \\
(5\;6)   &= \sigma_{1,3}^{(4)} \;(6\;7)\; \sigma_{1,3}^{(4)}, \\
(7\;8)   &= \sigma_{0,2}^{(3)} \;(6\;7)\; \sigma_{0,2}^{(3)}, \\
(8\;9)   &= \sigma_{1,2}^{(3)} \sigma_{0,1}^{(4)} \sigma_{0,2}^{(3)}
  \;(6\;7)\; \sigma_{0,2}^{(3)} \sigma_{0,1}^{(4)} \sigma_{1,2}^{(3)}, \\
(9\;10)  &= \left(\sigma_{0,1}^{(2)} \sigma_{0,2}^{(3)}\right)^2
  \;(6\;7)\; \left(\sigma_{0,2}^{(3)} \sigma_{0,1}^{(2)}\right)^2, \\
(10\;11) &= \sigma_{0,2}^{(4)} \sigma_{0,2}^{(3)} \sigma_{0,1}^{(2)}
  \sigma_{0,2}^{(3)} \;(6\;7)\; \sigma_{0,2}^{(3)} \sigma_{0,1}^{(2)}
  \sigma_{0,2}^{(3)} \sigma_{0,2}^{(4)}.
\end{align*}
These identities can all be verified computationally.

Now suppose $k \ge 4$ and $\CT_{(k)} = S_{L_k}$. If $L_{k+1} = \lcm(2, \dots,
k+1)$ equals $L_k$, then $\CT_{(k)}$ and $\CT_{(k+1)}$ act on the same set $\{0,
\dots, L_k - 1\}$ and every generator of $\CT_{(k)}$ is also a generator of
$\CT_{(k+1)}$, so
\[
S_{L_k} = \CT_{(k)} \le \CT_{(k+1)} \le S_{L_{k+1}} = S_{L_k},
\]
forcing $\CT_{(k+1)} = S_{L_{k+1}}$.

We may therefore assume $L_{k+1} \neq L_k$, in which case $k+1$ must be a prime
power $p^m$, and $L_{k+1} = p L_k$. We think of $\{0, \dots, L_{k+1} - 1\}$ as a
grid with $p$ rows of $L_k$ columns each: the element $n$ lies in row $\lfloor n /
L_k \rfloor$ and column $n \bmod L_k$. Since every element of $\CT_{(k)}$ is
$L_k$-periodic, it fixes each row and permutes the $L_k$ columns the same way in
every row; thus $\CT_{(k)} = S_{L_k}$ acts on the grid by these synchronised
column permutations. The inductive hypothesis implies that $\CT_{(k+1)}$
contains every permutation that acts the same on each row; we call these
permutations \emph{synchronised permutations}.

Our tool is Jordan's theorem on primitive permutations: a primitive permutation
group on $n \ge q + 3$ elements that contains a $q$-cycle (with $q$ prime)
contains the alternating group $A_n$. We apply it with the prime $q = 3$, so it
suffices to show that $\CT_{(k+1)}$ acts primitively and contains a $3$-cycle.
We establish these two facts in the next two lemmas.

\begin{lemma}
The group $\CT_{(k+1)}$ acts primitively on $\{0, \dots, L_{k+1} - 1\}$.
\end{lemma}
\begin{proof}
To show primitivity, we prove that $\CT_{(k+1)}$ is transitive and preserves
only the trivial partitions. Since $\CT_{(k)} = S_{L_k}$ acts as the full
symmetric group on the $L_k$ columns of each row, it suffices to show that the
modulus-$(k+1)$ generators extend these properties to the whole grid. We first
consider transitivity. For $0 < r < p$, the elements $rL_k - 1$ and $rL_k$ are
the last element of row $r - 1$ and the first element of row $r$ respectively.
We claim that they lie in the same $(k+1)$-block, that is $\lfloor (rL_k - 1) /
(k+1) \rfloor = \lfloor rL_k / (k+1) \rfloor$. Indeed, if this were not true,
then $k+1 \mid rL_k$, which is impossible as $k+1$ is a power of $p$ that does
not divide $L_k$, while $0 < r < p$. As the two elements lie in the same
$(k+1)$-block, there exists an element of $\CT_{k+1}$ that swaps $rL_k - 1$ and
$rL_k$, showing that $\CT_{(k+1)}$ acts transitively.

Let $\sim$ be an equivalence relation invariant under $\CT_{(k+1)}$, and let $B$
be a nontrivial equivalence class. By transitivity we may assume $0 \in B$ and
then suppose $0 \sim a$ with $a \neq 0$. We show that this block must contain
the entire grid. First, suppose that $a$ is not in the first column. Let $2 \le
c < L_k$ be a column other than the first two and different from the column that
contains $a$. Let $\sigma_1$ be the synchronised permutation swapping columns
$c$ and $a \bmod L_k$, and let $\sigma_2$ be the synchronised permutation
swapping columns $0$ and $1$. Now $\sigma_1$ fixes $0$ and sends $a$ to an
element $b$ in the same row as $a$ but in column $c$, which implies that $0 \sim
b$ and $b \in B$. Then $\sigma_2 \sigma_1$ sends $0$ to $1$ and $a$ to $b$, so
$1 \sim b$ and $1 \in B$ as well. Alternatively, if $a$ is in column $0$ we have
$L_k$ divides $a$. Then $k+1$ cannot divide $a$, since this would force $a = 0$.
Let $0 < c < k+1$ be such that $c \neq a \bmod k+1$. Then $\sigma_{0,c}^{(k+1)}$
sends $0$ to $c$ and fixes $a$, since it only acts on elements congruent to $0$
or $c$ modulo $k+1$. Hence $c \sim a$ and so $c \in B$ is in row $0$. In either
case, $0$ belongs to the same block as another element of the first row, so by
the primitivity of $\CT_{(k)}$ we conclude that $B$ contains the entire first
row.

To extend this to the entire grid, we first show that if the last element of a
row is in $B$, then the first element of the next row is also in $B$.
Specifically, let $0 < r < p$ be a nonzero row and we want to show $rL_k - 1
\sim rL_k$. Set $r_1 := (rL_k - 1) \bmod (k+1)$ and $r_2 := rL_k \bmod (k+1)$ -
we consider the action of $\sigma_{r_1, r_2}^{(k+1)}$. Since $r_2 \equiv r_1 + 1
\pmod{k+1}$, the residues $r_1$ and $r_2$ are distinct and $\sigma_{r_1,
r_2}^{(k+1)}$ swaps $rL_k - 1$ with $rL_k$. Now $0$ is not necessarily fixed by
$\sigma_{r_1, r_2}^{(k+1)}$, but because $r_1, r_2 < k+1 \le L_k$, $\sigma_{r_1,
r_2}^{(k+1)}(0)$ remains in row $0$ and therefore belongs to $B$. Consequently,
$0 \sim \sigma_{r_1, r_2}^{(k+1)}(0)$ and $rL_k - 1 \in B$ implies $rL_k \in B$.
Next, the synchronised permutation swapping column $0$ with column $j$ sends $0$
to $j$ and $rL_k$ to $rL_k + j$, which means that if $rL_k \in B$ then the
entire $r$-th row belongs to $B$. Working inductively, the entire space belongs
to block $B$, and therefore $\CT_{(k+1)}$ acts primitively on $\{0, \dots,
L_{k+1} - 1\}$.
\end{proof}

\begin{lemma}
The group $\CT_{(k+1)}$ contains a $3$-cycle.
\end{lemma}
\begin{proof}
Let $\sigma := \sigma_{0,1}^{(k+1)}$, let $\tau$ be the synchronised permutation
swapping column $0$ with column $2$, and set $g := [\sigma, \tau]$. A short
computation gives $g(0) = 1$, $g(1) = 2$, and $g(2) = 0$, so $g$ contains the
$3$-cycle $(0\;1\;2)$. We claim that if $m > 1$ (that is, if $k+1$ is composite)
then $g$ fixes every other point, since outside of $\{0, 1, 2\}$, the
transpositions of $\sigma$ and $\tau$ are disjoint. Suppose for contradiction
that some $z > 2$ is moved by both $\sigma$ and $\tau$. Then $z \bmod (k+1) \in
\{0, 1\}$ and $z \bmod L_k \in \{0, 2\}$, so $z \equiv a \pmod{k+1}$ and $z
\equiv b \pmod{L_k}$ for some $a \in \{0, 1\}$ and $b \in \{0, 2\}$. Such a $z$
exists if and only if $\gcd(k+1, L_k) = p^{m-1}$ divides $a - b$. If $a = b$,
then $a = b = 0$, so $z \equiv 0$ modulo both $k+1$ and $L_k$, hence modulo
$L_{k+1}$; this forces $z = 0$, contrary to $z > 2$. Otherwise $1 \le \lvert a -
b \rvert \le 2$, and $p^{m-1} \ge p \ge 2$ divides $a - b$ only when $p^{m-1} =
2$, i.e.\ $p = m = 2$; but then $k = p^m - 1 = 3$, contradicting $k \ge 4$.

Now suppose that $k+1 = p$ is prime. Let $\rho$ be the synchronised permutation
swapping column $5$ with column $L_k - 1$ and $h := \rho g \rho$. We claim that
$g' := g h^2$ is a $3$-cycle. First, $\sigma$ and $\tau$ are products of
disjoint transpositions, and because $\sigma$ moves the residues $0, 1 \bmod
(k+1)$ while $\tau$ moves the columns $0, 2$ (which differ by $2$), each
transposition of one shares a point with at most one transposition of the other.
Every orbit of $\sigma\tau$ therefore has at most three points, so $\sigma \tau$
is a product of disjoint $2$- and $3$-cycles. Hence $g^3 = ((\sigma \tau)^2)^3 =
(\sigma \tau)^6 = 1$, so $g$ is a product of $3$-cycles. Since $k+1$ is prime,
$k+1$ and $L_k$ are coprime, so by the Chinese remainder theorem let $0 \le A <
L_{k+1}$ be the value with $A \bmod k+1 = 3$ and $A \bmod L_k = 2$. Then let $B
:= A - 3$ and $C := A - L_k + 3$. Then $\tau(A) = A - 2$, $\sigma(A - 2) = B$,
and $\tau(B) = B$, so $g(A) = A - 2$. Similarly, $\tau(A - 2) = A$ and
$\sigma(A) = A$ so $g(A - 2) = B$, $\tau(B) = B$ so $g(B) = A$, and $\tau$ fixes
both $C$ and $\sigma(C)$ so $g(C) = C$. Finally, $\rho$ swaps $B$ and $C$, but
fixes $A$ and $A - 2$. Hence, $h$ acts as
\[
h(A) = A - 2, \qquad h(A - 2) = C, \qquad h(B) = B, \qquad h(C) = A
\]
and $g'$ acts as
\[
g'(A) = C, \qquad g'(B) = A, \qquad g'(C) = B.
\]
Therefore we have a $3$-cycle $(C\;B\;A)$, and it suffices to show $g'$ fixes
all other elements. Let $0 \le n < L_{k+1}$ with $n \neq A, B, C$. The key fact
is that the only element moved by both $\rho$ and $g$ is $B$. This follows from
the statement that if $\tau$ fixes both $m$ and $\sigma(m)$, then $g$ fixes $m$
as well. So if $\rho(m) \neq m$, then $m$ is in column $5$ or column $L_k-1$,
but $\tau$ only moves elements in columns $0$ and $2$, implying $\tau(m) = m$.
Since $\sigma(m)$ is one of $m-1, m, m+1$, the only way for $\sigma(m)$ to be
moved by $\tau$ is if $m \bmod L_k = L_k-1$ and $\sigma(m) = m + 1$. But
$\sigma(m) = m + 1$ means $m \bmod (k+1) = 0$. By the Chinese remainder theorem,
the equations $m \bmod L_k = L_k-1$ and $m \bmod (k+1) = 0$ uniquely describe a
value $0 \le m < L_{k+1}$, which is $B$.

First, suppose that $g(n)=n$. Since $g' = g\rho g^2\rho$, it suffices to show
that $g$ fixes $\rho(n)$. If $\rho(n)=n$, this is clear. Otherwise,
$\rho(\rho(n))=n\neq \rho(n)$, so $\rho$ moves $\rho(n)$. Moreover, $\rho(n)\neq
B$, since otherwise $n=\rho(B)=C$, contrary to our choice of $n$. By the key
fact, $g$ fixes $\rho(n)$. Now consider the case where $g(n) \neq n$. Since $n
\neq B$, the statement implies $\rho(n) = n$. Similarly, $g$ does not fix
$g^2(n)$ or else $g(n) = g^4(n) = g^2(n)$ which contradicts $g(n) \neq n$. We
also have that $g^2(n) \neq B$ or else $n = g(B) = A$, so we can conclude $\rho$
fixes $g^2(n)$. Therefore,
\[
g'(n) = g\rho g^2\rho(n) = g\rho g^2(n) = g^3(n) = n
\]
as desired, and in either case $\CT_{(k+1)}$ contains a $3$-cycle.
\end{proof}

We can now complete the inductive step. By the two lemmas, $\CT_{(k+1)}$ is
primitive and contains a $3$-cycle, so Jordan's theorem (with $q = 3$ and $n =
L_{k+1} \ge 6$) gives $A_{L_{k+1}} \le \CT_{(k+1)}$. To upgrade $A_{L_{k+1}}$ to
$S_{L_{k+1}}$, we exhibit an odd permutation. Let $d$ be the largest power of
$2$ dividing $L_{k+1}$; then $2 \le d \le k+1$, the quotient $L_{k+1} / d$ is
odd, and $\sigma_{0,1}^{(d)} \in \CT_{(k+1)}$ is a product of $L_{k+1} / d$
transpositions, of sign $(-1)^{L_{k+1}/d} = -1$. Hence $\CT_{(k+1)} =
S_{L_{k+1}}$, completing the induction.
\end{proof}

\section{Problem 21.24: Cograph power graphs are chordal}
\label{sec:2124}

There are many graphs naturally associated with a group, and these graphs often
illuminate group properties that are not immediately obvious from the algebraic
structure alone. For a finite group $G$, one such graph is the \emph{power
graph} $\mathcal{P}(G)$, defined as the undirected graph with vertex set $G$ in
which distinct vertices $x$ and $y$ are adjacent if either $x \in \langle
y\rangle$ or $y \in \langle x\rangle$. A similar notion is the so-called
\emph{enhanced power graph}, in which distinct vertices $x$ and $y$ are adjacent
if $\langle x,y\rangle$ is cyclic. With these definitions we note that every
power graph is a spanning subgraph of the corresponding enhanced power graph,
although the two graphs need not coincide.

Two properties that graphs can have are being a \emph{cograph}, meaning having
no induced path on four vertices ($P_4$-free), and being \emph{chordal}, meaning
having no induced cycle of length at least four ($C_n$-free for all $n \ge 4$).
Certainly, if a graph is a cograph, then it must be $C_n$-free for all $n \ge
5$. For arbitrary graphs, however, we cannot weaken the condition $n \ge 5$ to
$n \ge 4$. Indeed, neither property implies the other, as $P_4$ is chordal but
not a cograph, while $C_4$ is a cograph but not chordal.

On the other hand, for enhanced power graphs these properties are related.
Bubboloni, Fumagalli, and Praeger~\cite{BubboloniFumagalliPraeger} showed that
if the enhanced power graph of a finite group is a cograph, then it is chordal
as well. They then proposed Problem~21.24 of the Kourovka
Notebook~\cite{Kourovka}, conjecturing that the same statement holds for the
power graph.

In their paper they further made the related conjecture that for every finite
group $G$, if the power graph is a cograph, then the enhanced power graph is a
cograph too. They note that this is true for all finite simple groups, while the
power graph and enhanced power graph coincide in groups where every element has
prime power order.

As for the first conjecture, Cameron, Manna and
Mehatari~\cite{CameronMannaMehatari} classified the finite simple groups with
cograph power graph, while Brachter and Kaja~\cite{BrachterKaja} classified the
finite simple groups with chordal power graph. Combining these two
classifications implies that the first conjecture holds for all finite simple
groups. Here we show that it holds for all finite groups~\cite{KNote2124}, a
result which was independently obtained by Rundstr\"om~\cite{Rundstrom}.

\begin{theorem}
Let $G$ be a finite group. If $\mathcal{P}(G)$ is a cograph, then it is chordal.
\end{theorem}

\begin{proof}
As we alluded to before, any $C_n$ with $n \ge 5$ contains a $P_4$ by removing
all but four adjacent vertices. So it suffices to show that an induced $C_4$ in
the power graph will also imply the existence of an induced $P_4$.

Let $a$--$b$--$c$--$d$--$a$ be an induced $C_4$ in $\mathcal{P}(G)$, so that $a
\nsim c$ and $b \nsim d$. Consider the path $a$--$b$--$c$. Since $a \sim b$,
either $a \in \langle b\rangle$ or $b \in \langle a\rangle$ and because $b \sim
c$, either $b \in \langle c\rangle$ or $c \in \langle b\rangle$. If $a \in
\langle b\rangle$ and $b \in \langle c\rangle$, then there exist $i, j \in
\N$ such that $a = b^i = (c^j)^i = c^{i j}$, so $a \in \langle
c\rangle$, contradicting $a \nsim c$. Similarly, $b \in \langle a\rangle$ and $c
\in \langle b\rangle$ would give $c \in \langle a\rangle$, again a
contradiction. Therefore, either $a$ and $c$ both lie in $\langle b\rangle$, or
$b$ lies in both $\langle a\rangle$ and $\langle c\rangle$. The same argument
applied to the path $c$--$d$--$a$ yields that either $a$ and $c$ both lie in
$\langle d\rangle$ or $d$ lies in both $\langle a\rangle$ and $\langle
c\rangle$.

Combining these two dichotomies gives four cases. First, if $a, c \in \langle
b\rangle$ and $d \in \langle a\rangle \cap \langle c\rangle$, then $d \in
\langle a\rangle \le \langle b\rangle$, contradicting $b \nsim d$. Similarly, if
$a, c \in \langle d\rangle$ and $b \in \langle a\rangle \cap \langle c\rangle$,
then $b \in \langle c\rangle \le \langle d\rangle$, again contradicting $b \nsim
d$. We conclude that we either have $a, c \in \langle b\rangle \cap \langle
d\rangle$, or $b, d \in \langle a\rangle \cap \langle c\rangle$. The second case
works analogously to the first by splitting the cycle at $b$ rather than $a$, so
for the rest of the proof we assume without loss of generality that the first
case holds.

In this case, $a$ and $c$ both belong to the cyclic group $\langle b\rangle$,
while $a \notin \langle c\rangle$ and $c \notin \langle a\rangle$. In a cyclic
group, there is exactly one subgroup of each order dividing the group order. If
$\ord(a) \mid \ord(c)$, the unique subgroup of order $\ord(a)$ would sit inside
the unique subgroup of order $\ord(c)$, forcing $a \in \langle c\rangle$. Hence
$\ord(a) \nmid \ord(c)$, and by the same reasoning $\ord(c) \nmid \ord(a)$.
Therefore, there exist distinct primes $p$ and $q$ such that $p \mid \ord(a)$
and $q \mid \ord(c)$.

Since $b \notin \langle d\rangle$, we have $H := \langle b\rangle \cap \langle
d\rangle \subsetneq \langle b\rangle$ and $\lvert H\rvert < \ord(b)$. So there
also exists a prime $r$ with $v_r(\ord(b)) > v_r(\lvert H\rvert)$. Now fix $m_0
:= r^{v_r(\ord(b))}$, so that $m_0$ divides $\ord(b)$ but does not divide
$\lvert H\rvert$. If $r \neq p$ then $p \nmid m_0$ and $\ord(a) \nmid m_0$.
Similarly, if $r \neq q$ then $\ord(c) \nmid m_0$. Since $p$ and $q$ are
distinct, at least one of them is not equal to $r$. We can assume without loss
of generality that $r \neq p$; the other case is identical with $c$ in place of
$a$.

Since $\langle b\rangle$ is cyclic and $m_0 \mid \ord(b)$, there is an element
$b' \in \langle b\rangle$ of order $m_0$. We then claim that $b'$--$b$--$a$--$d$
is an induced $P_4$. We know $b \sim a$, $a \sim d$, and $b' \in \langle
b\rangle$, so $b' \sim b$. Therefore, $b'$--$b$--$a$--$d$ is a path. To prove
that it is also an induced path, we need to show these are the only relations,
that is $b \nsim d$, $b' \nsim d$ and $a \nsim b'$. As we already know $b \nsim
d$, we need to show $b' \notin \langle d\rangle$, $d \notin \langle b'\rangle$,
$a\notin\langle b'\rangle$ and $b'\notin\langle a\rangle$. Let us prove these in
order.

The first non-containment follows, as otherwise $b' \in H$ but $\ord(b') = m_0
\nmid \lvert H\rvert$. Secondly, $d \notin \langle b'\rangle$, since otherwise
$d \in \langle b\rangle$, contradicting $b \nsim d$. The third one follows from
the fact that $r\neq p \mid \ord(a)$, while $m_0$ is a power of $r$. For the
final one, we recall that $a\in H$, so $\ord(a)\mid \lvert H\rvert$, whereas
$m_0\nmid \lvert H\rvert$.
\end{proof}

\section{Problem 21.147: Right-relatively convex subgroups need not form a
sublattice}
\label{sec:21147}

We say that a group $G$ is \emph{right-orderable} if there exists a total order
$\le$ on $G$ such that $a \le b$ implies $ac \le bc$ for all $a, b, c \in G$.
Such an order is called a \emph{right-invariant total order}, or \emph{right
order} for short. A subgroup $H$ of a right-orderable group~$G$ is said to be
\emph{right-relatively convex} if it is convex with respect to some right order.
That is, $H$ is right-relatively convex if there exists a right order on $G$
such that, for all $a, b \in H$ and $c \in G$ with $a \le c \le b$, we have $c
\in H$.

A related concept is that of a \emph{lattice-ordered group}, or $\ell$-group for
short; an $\ell$-group is a group $G$ with a partial order $\le$ such that every
two elements $a,b \in G$ have both a greatest lower bound and a least upper
bound, and if $a \le b$ holds, then $cad \le cbd$ for all $c, d \in G$. Here
there is a similar notion of convexity; a subgroup $H$ of an $\ell$-group $G$ is
a \emph{convex $\ell$-subgroup} if it is closed under the two lattice operations
and if it is convex with respect to the order. For more information on
lattice-ordered groups, we refer the reader to the monograph by Kopytov and
Medvedev~\cite{KopytovMedvedev1994}.

For a given group $G$, we now consider the lattice of all subgroups of $G$. If
we take two elements $H, K$ of this lattice, we once again have two lattice
operations; we can take the intersection $H \cap K$ to get another subgroup, or
we can take the join $\langle H \cup K\rangle$, which is the smallest subgroup
that contains both $H$ and $K$.

With the above definitions in place, Conrad~\cite[Proposition 3.2]{Conrad1965}
proved that the convex $\ell$-subgroups of any $\ell$-group form a sublattice of
the subgroup lattice. Motivated by the analogy between lattice-ordered and
right-ordered groups (see e.g.~\cite{ColacitoMarra}), Bludov, Glass, Kopytov and
Medvedev~\cite[Problem 2.23]{BGKM} asked for which right-orderable groups the
analogous statement remains true when one replaces convex $\ell$-subgroups by
right-relatively convex subgroups. More concretely, Problem~21.147 of the
Kourovka Notebook~\cite{Kourovka}, attributed to V.~M.~Kopytov and
N.~Ya.~Medvedev and communicated by A.~V.~Zenkov, asks whether, for every
right-orderable group $G$, the set of right-relatively convex subgroups of $G$
forms a sublattice of its subgroup lattice.

We note that, if we fix a specific right order, then the corresponding convex
subgroups form a chain; for every two subgroups $H$ and $K$ which are both
convex according to the same order, we either have $H \subseteq K$ or $K
\subseteq H$. In particular, these subgroups do indeed form a sublattice. As it
turns out, however, this is no longer true for the set of all right-relatively
convex subgroups~\cite{KNote21147}.

\begin{theorem}
There exists a right-orderable group whose right-relatively convex subgroups
do not form a sublattice of its subgroup lattice.
\end{theorem}

\begin{proof}
We give a counterexample in the integer Heisenberg group
\[
G := \{a = (a_x, a_y, a_z): a_x, a_y, a_z \in \Z\}
\]
with group operation and inverses defined by
\[
a+b := (a_x + b_x,\;a_y + b_y,\;a_z + b_z + a_xb_y)
\qquad \text{and} \qquad
-a:=(-a_x,-a_y,-a_z+a_xa_y).
\]
Even though $G$ is non-abelian (indeed, one can check that e.g.\ $(1,0,0)$ and
$(0,1,0)$ do not commute), we will use additive notation throughout. The reason
for this is that most of the computations below take place in the abelian
subgroup of elements with second coordinate equal to $0$, in which $a+b$ always
has vanishing cross-term $a_xb_y$.

With the subgroups $H_1, H_2$ defined as $\langle(1,0,1)\rangle$ and
$\langle(1,0,-1)\rangle$ respectively, it is sufficient to show that these
subgroups are right-relatively convex, even though their join $J := \langle H_1,
H_2 \rangle$ in the subgroup lattice of $G$ is not.

As for the latter claim, with $e := (1, 0, 0)$ we have
\[
2e=(2,0,0)=(1,0,1)+(1,0,-1)\in J.
\]
On the other hand, any element $(r, 0, s) \in J$ has $r \equiv s \pmod{2}$,
implying $e \notin J$. Hence, since we either have $2e \le e \le 0$ or $0 \le e
\le 2e$ for every right order $\le$ on $G$, $J$ is not right-relatively convex.
It remains to show that $H_1$ and $H_2$ are right-relatively convex, and let us
start with $H_1$.

Define $P$ as the subset of all $a \in G$ such that either $a_y > 0$, or $a_y =
0$ and $a_x > a_z$, or $a_y = 0$ and $a_x = a_z > 0$. That is, $a \in P$
precisely when $(a_y, a_x - a_z, a_z)$ is larger than $(0, 0, 0)$ in the
lexicographic order. We then define the order $\le$ on $G$ as follows: $a \le b$
if either $a = b$ or $b - a \in P$. First, from the fact that $P$ is closed
under the group operation, it follows that the relation $\le$ is transitive.
Secondly, this gives a total order, because for every nonzero $a \in G$ exactly
one of $a$ and $-a$ lies in $P$. Finally, it is right-invariant, because $(b+c)
- (a+c) = b-a$. We aim to show that $H_1$ is convex with respect to this order.

So assume $a \le c \le b$ with $a, b \in H_1$. As $a$ and $b$ are elements of
$H_1$, we see that $a_x = a_z$, $b_x = b_z$ and $a_y = b_y = 0$. By combining
the latter equalities with $a \le c \le b$, it follows that $c_y = 0$ as well.
We then find that $c - a = (c_x - a_x, 0, c_z - a_x)$ and $b - c = (b_x - c_x,
0, b_x - c_z)$, which imply $c_x \ge c_z$ and $c_x \le c_z$ respectively. We
thereby conclude $c_x = c_z$ so that $c \in H_1$, as desired.

For $H_2$, the same argument works if we instead define $P$ as the subset of all
$a \in G$ such that $(a_y, a_x + a_z, a_z)$ is larger than $(0, 0, 0)$ in the
lexicographic order. This finishes the proof.
\end{proof}

We remark that this specific construction works more generally by taking the
semidirect product $G$ of a torsion-free abelian group $A$ with a
right-orderable group $H$. As long as $A$ has pure subgroups whose sum is not
pure (where a subgroup $B \le A$ is called \emph{pure} if $na \in B$ for some $a
\in A, n \in \N$ implies $a \in B$), the same obstruction applies.

Finally, one can note that the ambient group $G$ is essentially superfluous
here, and the same argument works just as well in $A$. The group $G$ is used
simply to show that counterexamples exist both in the abelian and non-abelian
case, and in the appendix we elaborate on this further.

\section{Problem 21.150: Counterexample to a rank inequality for
\texorpdfstring{$p$}{p}-group extensions}
\label{sec:21150}

For a prime $p$ we recall that an elementary abelian $p$-group is an abelian
group in which all non-identity elements have the same order $p$. In particular,
if such a group is finite and nontrivial, it is isomorphic to $(\Z/p\Z)^r$ for
some positive integer $r$, which we call its rank. Here we consider finite
groups $G$ which contain a normal subgroup $A$ and a subgroup $B$ such that both
$A$ and $B$ are elementary abelian $p$-subgroups of $G$, and for which the
natural map $B\to G/A$ is an isomorphism. Assume now that $A$ contains an
element $a$ that does not commute with any non-identity element of $B$.
Equivalently, assume that the centraliser $C_B(a)$ of $a$ in $B$ is trivial. We
now define $H := \langle a,B\rangle$, and let $Z(H)$ and $H'$ be the centre and
commutator subgroup of $H$ respectively. Problem~21.150 of the Kourovka
Notebook~\cite{Kourovka}, proposed by V.~I.~Zenkov, then asks whether one always
has $\rank\bigl(Z(H)\cap H'\bigr)\le\rank(B)$. As it turns
out~\cite{KNote21150}, the answer is no.

\begin{theorem}
There exists a $p$-group $G$, which is the extension of an elementary abelian
normal subgroup $A$ by an elementary abelian subgroup $B \cong G/A$ such that,
with $C_B(a) = 1$ for some $a \in A$ and $H := \langle a,B\rangle$, we have
\[
\rank\bigl(Z(H)\cap H'\bigr) > \rank(B).
\]
\end{theorem}

\begin{proof}
With $p = 3$ we let $A\cong(\Z/p\Z)^6$. Denoting $\mathfrak{m}:=(x,y)$, we then
identify $A$ with the additive group of the $6$-dimensional $\F_3$-algebra
$\F_3[x,y]/\mathfrak{m}^3$ on the basis $\{1,x,y,x^2,xy,y^2\}$. We further
define the functions $b_1$ and $b_2$ by
\[
b_1(u) := u(1+x) \qquad \text{and} \qquad b_2(u) := u(1+y),
\]
and define $B := \langle b_1,b_2\rangle \cong (\Z/p\Z)^2$, which has rank $2$.
Then $B$ acts on $A$ by multiplication, and we can define $G$ as the semidirect
product $A\rtimes B$. Letting $a$ be the constant element $1$ in $A$, we claim
that $C_B(a)$ is trivial. Indeed, we note that $b \in C_B(a)$ if and only if
$b(a) = a$. And given $b = b_1^i b_2^j \in B$, we have $b(a) = a$ precisely when
$(1 + x)^i (1 + y)^j = 1$, which implies $i = j = 0$. With $H := \langle
a,B\rangle$ it now suffices to show that $Z(H)\cap H'$ has rank at least $3$. In
fact, since $x^2,xy,y^2$ are $\F_3$-linearly independent in $A$, it is
sufficient to show that both $\{x^2,xy,y^2\} \subseteq Z(H)$ and $\{x^2,xy,y^2\}
\subseteq H'$.

The first containment is not hard to see, as multiplication by $1+x$ or $1+y$
acts trivially modulo $\mathfrak{m}^3$ on the degree-$2$ elements $x^2$, $xy$
and $y^2$. As for the second containment, computing commutators in $H$ gives
$x=[b_1,a]$ and $y=[b_2,a]$. Hence, by iterating we see that
\[
x^2=[b_1,[b_1,a]],\qquad
xy=[b_1,[b_2,a]],\qquad \text{and}\qquad
y^2=[b_2,[b_2,a]]
\]
all lie in $H'$, as desired.
\end{proof}

\section*{Acknowledgements}
We thank Khukhro and Mazurov for their work as editors of the Kourovka Notebook
and for sharing our submissions with the original authors of the problems and
other reviewers. We are further grateful to them for taking the time to examine
our solutions. We also thank Stefano Rocca for his valuable contributions to
this work.

\appendix

\section{Formal discovery and verification with Aristotle}
\label{sec:appendix}

As explained in the introduction, Aristotle, the formal reasoning agent
developed by Harmonic~\cite{Aristotle}, autonomously discovered formal solutions
to all eight problems. In this appendix, we describe the process in detail as a
case study of automated proof discovery in collaboration with human
mathematicians. We aim to give a transparent account of what was and was not
done by Aristotle, and to apply the same research standards to these results as
we would to any other mathematical work.

Before we could ask Aristotle to solve a problem from the Kourovka Notebook, we
had to decide \emph{which} problem we wanted it to tackle. Three considerations
guided the choice. The first was our mathematical judgement about which problems
seemed amenable to attack. When we began, we were not aware of any problem from
the Kourovka Notebook that had been resolved by an artificial intelligence, let
alone by one working in a formal language.\footnote{Our solutions to
Problems~21.150 and~20.125, submitted in March 2026, were the first of the
eight presented here, and the editors subsequently described them as the first
in the Notebook's history to have been obtained by an AI system:
\url{https://algebra-lincoln.org/2026/05/19/ai-solving-phd-level-problems-in-mathematics/}.}
It was therefore unclear whether we could expect to solve even one of them, and
we steered away from famous open problems in order to maximise the chance of
success. The second was our experience that one of Aristotle's notable strengths
is coming up with constructions, examples and counterexamples. This is reflected
in the choice of problems: more than half of the resulting solutions rest on
explicit group constructions. The third was \emph{Mathlib-adjacency}: how much
background theory Aristotle would have to develop before it could even state
the problem. A problem is Mathlib-adjacent when the notions it involves are
already formalised, so that its formal statement sits a short step from
existing definitions, whereas a problem far from Mathlib requires that
infrastructure to be built first, at a cost independent of the difficulty of
the mathematics.

On this last point, once we had chosen a problem for Aristotle to solve, the
initial step was to ask it to formalise the definitions needed to state the
problem in Lean. This was often necessary, either because a definition was
specific to the phrasing of the problem, as with permuted products in
Section~\ref{sec:1850}, or because the relevant notion, such as a Rota--Baxter
operator or power graph, was not yet present in Mathlib. We then asked Aristotle
to use these definitions to formulate the problem in Lean. In most cases,
Aristotle produced a correct formal statement that required little subsequent
human editing. When a statement was incorrect, this was generally because the
relevant notion admitted several possible definitions. Aristotle then either
chose a definition that was different from the intended one, or it combined
features of multiple definitions. However, even logically correct statements
sometimes required human revision. This is because mathematically equivalent
formulations can differ considerably in how convenient they are to use. For
example, one problem we tried Aristotle on involved the notion of the outer
automorphism group, a concept which is not yet available in Mathlib. However,
rather than set up the definitions needed to express the problem in terms of the
outer automorphism group, Aristotle expanded the statement directly into a
longer, lower-level formula, in this case asserting that every automorphism is
either a conjugation or a subgroup element composed with a conjugation. This is
logically equivalent to the intended statement, but is considerably less
convenient for developing and applying general theory.

After completing the initial set-up, we prompted Aristotle to search for a
solution or at least make partial progress. Some, like Problem~18.50 and
Problem~21.24, were resolved relatively quickly once the definitions and
statements were in place. For others, the progress was slower. A recurring
hurdle we had to overcome was Aristotle's propensity to give up on open
problems. For the author prompting Aristotle, this meant finding ways to ask for
partial progress. Some techniques included asking Aristotle to write a proof
outline first, turning challenging steps into intermediate lemmas, having
Aristotle focus on an alternative formulation of the problem that it could not
recognise as an open problem, or isolating key parts of the problem into
separate problems that could be solved independently. For example, in
Problem~21.8 Aristotle quickly found how to map class transpositions to finite
permutations and proved that the mapping was injective. This reduced the problem
to showing that the image of the class transpositions generates the entire
finite group. This isolated the core challenge of the problem and was less
recognisable to Aristotle as an open problem, which made it more willing to try
out various strategies. Over time, Aristotle was able to come up with the steps
needed for each solution and assemble the full proof.

Once a formal solution was found by Aristotle, we prepared a faithful
informalisation to be shared with Kourovka Notebook editors, problem proposers,
and reviewers. Since Aristotle works with formalised mathematics in Lean, the
first step is to convert the formal solution to an informal solution written in
natural language. As with the formal solution, the first pass was done by
Aristotle. Afterwards, we further refined the writing in order to make it read
more naturally and to include the right amount of detail a human reviewer would
expect.

After the conversion, we sought review from the problem authors whenever
possible, sometimes through the Kourovka Notebook editors. A few instances stand
out. For example, for Problem~21.24, the proposers informed us that Rundstr\"om
had independently solved the same problem and we agreed to submit them together.
Both are now credited in the Notebook.

For Problem~21.147, we recall that the example in Section~\ref{sec:21147} uses
the non-abelian integer Heisenberg group $G$. At first glance this construction
may seem overly complicated, as we only work in an abelian subgroup $A \le G$.
In fact, the original solution did indeed use $A$ instead. However, after we
sent it to the authors, they explained that this example was already known and
that they had meant to ask about a non-abelian example instead. Aristotle then
quickly realised that one could simply use a semidirect product construction to
embed the abelian group in a larger group and thereby obtain a non-abelian
counterexample as well. Although the distinction ultimately did not affect the
solution in this case, the episode illustrates a general phenomenon that often
arises when formalising mathematics: the discovery of hidden assumptions made by
mathematicians, often without realising it.

Another prominent example of this latter phenomenon concerns Problem~21.149 on
so-called Dlab groups~\cite{Dlab1968}. This problem originally asked: are there
order automorphisms of Dlab groups that are not inner automorphisms? After
Aristotle set up the necessary definitions, it managed to resolve this question
in the affirmative, and we subsequently shared the solution with the editors of
the Kourovka Notebook. Afterwards, the editors sent it to one of the problem
authors who reviewed it and agreed that it was correct. However, he then
explained that he actually meant to ask a stronger question. Indeed, what he
truly wanted to know was whether there exist order automorphisms of Dlab groups
that are not induced by conjugation by elements of a (possibly larger) Dlab
group. This more difficult problem is still open, and the problem statement in
the Kourovka Notebook has since been refined to reflect this intended meaning.
This episode showcases the role that formal mathematics can play in
clarification: it can not only certify arguments, but also sharpen the
statements themselves.

Once the solutions were reviewed and accepted by the editors of the Kourovka
Notebook (see \cite{KNote346, KNote1850, KNote1925, KNote20125, KNote218,
KNote2124, KNote21147, KNote21150} for the original submissions), we revised
both the informal and formal versions for publication. For the informal proofs,
this meant polishing the exposition and adding relevant mathematical context and
references. For the formal proofs, this included improving documentation,
organisation and style, in order to make the proofs easier to read. It further
included improving performance and adhering to usual code standards to increase
reusability. This is important, as these documents do not stand alone, but are
part of a broader picture. Indeed, the resulting proofs and constructions might
be applicable to other problems, while important definitions and theorems will
be added to Mathlib to extend the standard library for everyone to use. This, in
turn, expands the mathematical infrastructure available to Aristotle, creating a
positive feedback loop that benefits everyone.

More generally, as authors we believe in the potential of artificial
intelligence as a collaborative tool for mathematicians. We hope this work
demonstrates that responsible use in this domain means maintaining rigorous and
verifiable proof standards, transparently documenting the human-AI workflow,
actively communicating findings to foster new collaborations with other
researchers, and contributing foundational infrastructure to public, open source
libraries such as Mathlib.

\end{document}